\documentclass[12pt]{article}
\usepackage{amsmath,amsthm,amsfonts,amssymb,csquotes,epsfig, titlesec, epsfig, float, caption,wrapfig, mathrsfs, mathtools, multicol}
\usepackage[title]{appendix}
\usepackage{authblk}
\usepackage[hidelinks]{hyperref}
\usepackage{tikz}
\usepackage{tikz-cd}
\usetikzlibrary{cd,calc}
\usetikzlibrary{decorations.markings}
\usetikzlibrary{decorations.pathreplacing,shapes.misc}
\usetikzlibrary{decorations.pathmorphing}
\usepackage{xcolor}
\usepackage{accents}
\usepackage{marvosym}
\usepackage[alphabetic,abbrev,lite,backrefs]{amsrefs} 
\usepackage[]{enumerate}
\usepackage[american]{babel}
\usepackage[left=1in,top=1in,right=1in]{geometry}
\usepackage[capitalize]{cleveref}
\usepackage{subcaption}
\usepackage[all]{xy}
\usepackage{enumitem}
\usepackage{comment}

\newtheorem{mainthm}{Theorem}
\Crefname{mainthm}{Theorem}{Theorems}
\newtheorem{maincor}[mainthm]{Corollary}
\Crefname{maincor}{Corollary}{Corollaries}

\usepackage{chngcntr}
\counterwithin{figure}{section}
\numberwithin{equation}{section}

\newtheorem{lemma}[equation]{Lemma}

\newtheorem{prop}[equation]{Proposition}
\newtheorem{cor}[equation]{Corollary}
\newtheorem{thm}[equation]{Theorem}

\theoremstyle{definition}

\newtheorem{examplex}[equation]{Example}
\newenvironment{example}
    {\pushQED{\qed}\examplex}
  {\popQED\endexamplex}

\theoremstyle{remark}
\newtheorem{remarkx}[equation]{Remark}
\newenvironment{remark}
    {\pushQED{\qed}\remarkx}
  {\popQED\endremarkx}
\newtheorem*{case*}{Case}

\titleformat*{\section}{\normalsize \bfseries \filcenter}
\titleformat*{\subsection}{\normalsize \bfseries }
\newcommand{\Manoa}{M\=anoa}
\newcommand{\Hawaii}{Hawai\kern.05em`\kern.05em\relax i}

\renewcommand{\geq}{\geqslant}
\renewcommand{\leq}{\leqslant}

\newcommand{\beq}{\begin{displaymath}}
\newcommand{\eeq}{\end{displaymath}}

\newcommand{\coloneq}{\ensuremath{\mathrel{\mathop :}=}}

\newcommand{\PP}{\mathbb P}
\newcommand{\QQ}{\mathbb Q}

\newcommand{\RR}{\ensuremath{\mathbb{R}}}
\newcommand{\ZZ}{\mathbb Z}

\newcommand{\cA}{\mathcal{A}}
\newcommand{\cB}{\mathcal B}

\newcommand{\cE}{\mathcal{E}}

\newcommand{\cO}{{\mathcal O}}

\newcommand{\sX}{{\mathcal X}}

\newcommand\wt{\widetilde}

\newcommand\tsX{\widetilde{\sX}}

\newcommand{\Cl}{\on{Cl}}

\newcommand\Coh{\operatorname{Coh}}

\newcommand{\Cox}{\operatorname{Cox}}

\newcommand{\Hom}{\operatorname{Hom}}

\newcommand{\leftexp}[2]{{\vphantom{#2}}^{#1}{#2}}

\newcommand\on{\operatorname}

\newcommand{\weezer}{\leftexp{=}{\kern-0.23em\operatorname{W}}^{\kern-0.21em =}}

\newcommand{\Aut}{\operatorname{Aut}}

\title{\large \textbf{The geometric Merkurjev-Panin Conjecture for the Cox category}}
\date{}
\author{{\normalsize Daniel Erman, Andrew Hanlon, Gaku Liu, and Hailun Zheng}}

\newcommand{\Addresses}{{
  \bigskip
  \footnotesize

  \noindent D.~Erman, \textsc{Department of Mathematics, University of \Hawaii \, at \Manoa}\par\nopagebreak
  \noindent \textit{E-mail address}: \texttt{erman@hawaii.edu}

  \medskip
  
  \noindent A.~Hanlon, \textsc{Department of Mathematics, University of Oregon}\par\nopagebreak
  \noindent \textit{E-mail address}: \texttt{ahanlon@uoregon.edu}

  \medskip

  \noindent G.~Liu, \textsc{Department of Mathematics, University of Washington}\par\nopagebreak
  \noindent \textit{E-mail address}: \texttt{gakuliu@uw.edu}

  \medskip

  \noindent H.~Zheng, \textsc{Department of Mathematics, University of \Hawaii \, at \Manoa}\par\nopagebreak
  \noindent \textit{E-mail address}: \texttt{hailunz@hawaii.edu}
  
}}

\begin{document}
\maketitle

\begin{abstract}
    We show that a strong version of the geometric Merkurjev-Panin conjecture holds for the Cox category of a projective toric variety.
    That is, we prove that the full strong exceptional collection of Bondal-Thomsen line bundles is invariant under the group of lattice automorphisms that permute the rays of the toric variety's fan.
    Our result is meant to further illustrate that the Cox category is a natural repository for homological algebra on toric varieties.
\end{abstract}

\section{Introduction}
 A well-known question of Merkurjev and Panin (see the Question immediately after \cite[Corollary 7.8]{MP}) asks if the Grothendieck group of the separable closure of an arithmetic toric variety is a permutation module for the Galois group action.
A geometric generalization of Merkurjev and Panin's question, which initially appears in print in the introduction of Castravet and Tevelev's \cite{CT1}, asks the following: Let $X$ be a smooth, projective toric variety and let $G$ be a finite subgroup of the toric automorphism group of $X$. Is $K_0(X)$ a permutation $G$-module?  In fact, they raised an even stronger possibility:
\begin{quote}
    {\em [W]e do not know any smooth projective toric varieties $X$ with an action of a finite group $\Gamma$ normalizing the torus action
which do not have a $\Gamma$-equivariant exceptional collection $\{E_i\}$ of maximal
possible length.}  \cite[Page 6]{CT1}
\end{quote}
In other words, in the context of this setup, Castravet and Tevelev ask whether $X$ admits a $G$-invariant full, exceptional collection.
Returning to the arithmetic context, the categorification of the Merkurjev-Panin question was also raised in \cite[Question 1.5]{BDMatv}.

We will prove that the conjecture becomes true--and in the strongest possible sense--if one passes from the bounded derived category $D(X)$ to the Cox category of $X$, as defined in \cite{king}.  The definition of $D_{\Cox}(X)$ is given in \cref{sec:background}, but we recall that it is built from the secondary fan of $X$, and it comes with a natural fully faithful map $D(X)\to D_{\Cox}(X)$.

\begin{mainthm}\label{thm:main}
The Bondal-Thomsen collection is a $G$-invariant full strong exceptional collection for $D_{\Cox}(X)$.    In particular, if we write $K_{\Cox}(X)$ for the Grothendieck group of $D_{\Cox}(X)$, then $K_{\Cox}(X)$ is a permutation $G$-module with the classes of the Bondal-Thomsen collection providing the corresponding $G$-invariant basis.
\end{mainthm}

Of course, in the rare cases where the Bondal-Thomsen is exceptional on $X$ (see \cite{Bondal} for geometric conditions), the Bondal-Thomsen collection is a $G$-invariant full exceptional collection for $D(X)$.
This follows from \cref{thm:main}, but it also immediately follows from the fact that the Bondal-Thomsen collection is the set of all possible summands of the pushforward of the structure sheaf under toric Frobenius \cite{Thomsen} as the toric Frobenius morphism commutes with all toric morphisms.
In some cases, a subset of Bondal-Thomsen is $G$-invariant and full as is the case for toric Fano three-folds by \cite{uehara}, some toric Fano four-folds by \cite{prabhu-naik}, and a smattering of other cases observed in \cite{BDMatv}.
The geometric Merkurjev-Panin question is also answered affirmatively for the permutohedral toric varieties in \cite{CT1} by giving a full $G$-invariant exceptional collection of sheaves. 
In \cite{sanchez2023derived}, Sanchez strengthens the result of \cite{CT1} giving a full strong $G$-invariant exceptional collection of line bundles for the permutohedral varieties, stellahedral varieties, and type $B_n$ Coxeter varieties.
Note that Sanchez's exceptional collections do not coincide or are even subsets of the Bondal-Thomsen collection as noted in \cref{rem:Sanchezcompare}.
As far as we are aware, the references above encompass all cases where $D(X)$ is known to admit a full $G$-invariant exceptional collection. 

\cref{thm:main} provides a new proof of one of the main results of~\cite{MP}.  Namely, in \cite[Corollary 7.8]{MP}, Merkurjev and Panin show that the Grothendieck group of the separable closure of an arithmetic toric variety is a direct summand of a permutation module over the Galois group. In our setting, since $D(X)$ embeds fully faithfully in $D_{\Cox}(X)$ as an admissible subcategory, we obtain the geometric generalization of their theorem as an immediate consequence of \cref{thm:main}.

\begin{maincor}\label{cor:summand} The Grothendieck group $K_0(X)$ is a summand of a permutation $G$-module, namely $K_{\Cox}(X)$. 
\end{maincor}

Perhaps the primary interest in \cref{thm:main} is that it further underscores the philosophy of \cite{king}.  Over the past several decades, when studying sheaves on a variety it has become increasingly standard to pass from $\Coh(X)$ to $D(X)$ for several reasons, including: $D(X)$ tends to have nicer properties, and yet little is lost because $\Coh(X)$ has a fully faithful embedding into $D(X)$.
In~\cite{king}, the authors go one step further for toric varieties, replacing $D(X)$ by a category $D_{\Cox}(X)$ that sees all of the varieties the from secondary fan of $X$.  The core observation is that $D_{\Cox}(X)$ has even nicer properties than either $\Coh(X)$ or $D(X)$, and  both of these have fully faithful embeddings into $D_{\Cox}(X)$.  In~\cite{king}, the focus was on showing that the Bondal-Thomsen collection provides a full, strong exceptional collection of line bundles, thus realizing King's Conjecture (which was known to be false for $D(X)$ in general).

Theorem~\ref{thm:main} can be seen in a similar light, showing that Castravet and Tevelev's question not only has a positive answer for $D_{\Cox}(X)$, but that in fact, the resulting exceptional collection is a simple and natural set of line bundles.

\begin{remark}
    We have chosen $X$ to be smooth and projective so far only for simplicity of the exposition. 
    Appropriately modified, \cref{thm:main} holds when $X$ is normal and semiprojective (see \cref{thm:generalMain}) and \cref{cor:summand} holds when $X$ is simplicial and semiprojective (see \cref{cor:generalcorB}).
    When working with simplicial $X$, the only modification is that $D(X)$ does not embed fully faithfully into $D_{\Cox}(X)$. However, if $\sX$ is the corresponding smooth Deligne-Mumford stack, $D(\sX)$ fully faithfully embeds into $D_{\Cox}(X)$. When $X$ is semiprojective but not projective, the Bondal-Thomsen collection is not an exceptional collection but rather consists of the summands of a tilting object of $D_{\Cox}(X)$.

    In fact, the Cox category construction depends only on the rays of the fan of $X$ or equivalently its Cox ring. Thus, even if $X$ is not simplicial, we can still associate a Cox category with a $G$-invariant Bondal-Thomsen collection.
    In this case, $D_{\Cox}(X)$ is a noncommutative resolution of $X$ \cite[Section 7.3]{king}.

    Further, we have also simplified the narrative in this introductory section by taking $G$ to be the set of toric automorphisms of $X$. Below, we will show that \cref{thm:main} holds for the larger group of automorphisms of the cocharacter lattice that preserve only the rays of the fan of $X$.
\end{remark}

The essence of proving Theorem~\ref{thm:main} is to show that the $G$-action on $X$ can be naturally extended to a $G$-action on $D_{\Cox}(X)$ that permutes the Bondal-Thomsen collection $\Theta$.  This requires two significant steps.  
\begin{enumerate}
    \item We show that $G$ acts naturally on the secondary fan of $X$, thus inducing an action on the set of toric stacks $\{\sX_1, \dots, \sX_r\}$ corresponding to the maximal chambers of that secondary fan.
    \item We show that one can find a $G$-equivariant simplicial refinement of all of the stacky fans corresponding to $\sX_1, \dots, \sX_r$. 
\end{enumerate}  
The second step allows us to produce a $G$-equivariant, simplicial toric cover $\pi_i\colon \widetilde{\sX}\to \sX_i$ for all $1\leq i \leq r$, which in turns allows us to lift the $G$-action on the secondary fan to the Cox category $D_{\Cox}(X) \subseteq D(\widetilde{\sX})$ by restriction.

As an additional corollary, we prove that the resolution of a toric subvariety $Y\subseteq X$ from \cite{HHL} inherits the equivariant properties of $Y$ in $X$.
See \cref{prop:ginvHHL} for the precise statement.  
In order to obtain this corollary, we describe the $G$-action geometrically in terms of the Bondal stratification, which also makes it easier to understand the $G$ action on $\Theta$ in examples.

\medskip

This paper is organized as follows.
In \cref{sec:background}, we summarize some essential background material.  
\cref{sec:equivariant cover} is focused on the construction of the equivariant cover $\widetilde{\sX}$, which we use in \cref{sec:proofOfMain} to prove our main results.  
In \cref{sec:furtherInvariant}, we relate the $G$-action on $\Theta$ to the Bondal stratification.
\cref{sec:permex} examines the example of a permutahedral variety in some detail.

\subsection{Acknowledgements} We thank Jenia Tevelev for introducing us to the geometric Merkurjev-Panin question and pointing out that \cite{king} is closely related.
We are also grateful to Chris Eur and Alex Fink for explaining aspects of the permutohedral variety.
All of these conversations took place at the Oberwolfach workshop``Toric Geometry" in April 2025, and we thank Oberwolfach for their hospitality and stimulating research environment.
We also thank Matthew R.\ Ballard, Christine Berkesch, Michael K. Brown, Lauren Cranton Heller, David Favero, Sheel Ganatra, Jeff Hicks, Jesse Huang, Oleg Lazarev, and Nicholas Proudfoot for useful conversations. 

D. E. was supported by NSF grants DMS-2200469, DMS-2200469, and DMS-2412044.
A. H. was supported by NSF grants DMS-2549013 (previously as DMS-2412043) and DMS-2549038. 
G. L. was supported by NSF grant DMS-2348785.
H. Z. was supported by NSF Grant DMS-2535689.

\section{Background and setup} \label{sec:background}

In this section, we set our notation more precisely and recall background on toric stacks and the Cox category. We will be rather terse on both fronts and refer to \cite{BCS,GS} for more details on the former and \cite{king} for the latter.

We work over an algebraically closed field $\Bbbk$. We let $N$ be a lattice, $M$ the dual lattice, and $N_\RR \coloneq N \otimes_\ZZ \RR$.
We will use the last notation for the tensor product more generally.
Given a fan $\Sigma$, we let $\Sigma(k)$ be the set of $k$-dimensional cones in $\Sigma$. 
For $\rho \in \Sigma(1)$, we let $u_\rho$ denote the primitive generator of $\rho$. 
We let $G \subset \Aut(N)$ be the group of automorphisms preserving $\Sigma(1)$ as a set.

We will work with smooth toric Deligne-Mumford stacks in the sense of \cite{BCS}.
A smooth toric Deligne-Mumford stack is determined by a simplicial fan $\Sigma$ on $N_\RR$ and a homomorphism $\beta \colon \ZZ^{\Sigma(1)} \to N$ such that for every standard basis vector $e_\rho \in \ZZ^{\Sigma(1)}$, $\beta(e_\rho) = b_\rho u_\rho$ for some $b_\rho > 0$ \cite[Section 3]{BCS}. 
The data of $(\Sigma, \beta)$ is called a stacky fan, and is generalized in \cite{GS}.
Given a stacky fan $(\Sigma, \beta)$, we let $\sX_{\Sigma, \beta}$ be the corresponding smooth Deligne-Mumford stack.
By \cite[Remark 4.5]{BCS} and \cite[Section 3]{GS} more generally, a commutative diagram
\[ \begin{tikzcd}
\ZZ^{\Sigma(1)} \arrow{r}{\widehat{f}} \arrow{d}{\beta} & \ZZ^{\Sigma'(1)} \arrow{d}{\beta'} \\
N \arrow{r}{f} &  N' 
\end{tikzcd}\]
where $f$ is map of fans from $\Sigma$ to $\Sigma'$ gives rise to a unique toric morphism $\phi_f \colon \sX_{\Sigma, \beta} \to \sX_{\Sigma', \beta'}$.

We now turn to the Cox category. 
We fix a semiprojective toric variety $X$ with cocharacter lattice $N$ and fan $\Sigma$. 
We let $\Cl(X)$ denote the class group of $X$, and for any $\rho \in \Sigma(1)$, we let $D_\rho$ be the corresponding torus-invariant divisor.
Without loss of generality and for simplicity, we can assume that $\Sigma$ is simplicial.
In fact, our input data for the following construction is really only the rays of $\Sigma(1)$ or equivalently the $\Cl(X)$-graded Cox ring $S$ of $X$.

We write $\Sigma_{GKZ}(X)$ for the secondary fan of $X$; see~\cite[Chapter 14]{CLSToricVarieties} for a thorough review.  We write $\Gamma_1, \dots, \Gamma_r$ for the maximal chambers of $\Sigma_{GKZ}(X)$.  Each such maximal chamber corresponds to a simplicial fan $\Sigma_i$ with corresponding toric variety $X_i$. 
We write $\sX_i$ for the corresponding smooth toric DM stack corresponding to $(\Sigma_i, \beta_i)$ where $\beta_i \colon \ZZ^{\Sigma_i(1)} \to N$ satisfies $\beta_i(e_\rho) = u_\rho$ for all $\rho \in \Sigma_i(1)$.
We choose a smooth toric DM Stack $\tsX$ with proper birational maps $\pi_i \colon \tsX \to \sX_i$ given by refinements of stacky fans. 
Such a $\tsX$ exists by \cite[Section 3.1]{king}.
The Cox category $D_{\Cox}(X) = D_{\Cox}(S)$ is defined to be the triangulated hull of the $\pi_i^* D(\sX_i)$ inside of $D(\tsX)$. 
By \cite[Corollary 3.9]{king}, $D_{\Cox}(X)$ does not depend on the choice of $\tsX$. 
We note that the the functors $\pi_i^*\colon D(\sX_i)\to D_{\Cox}(X)$ are fully faithful for all $i$ and their adjoints $\pi_{i\ast} \colon D_{\Cox}(X) \to D(\sX_i)$ are localizations. 

The Bondal-Thomsen collection $\Theta \subset \Cl(X)$ is the set of degrees that are linearly equivalent (over $M$) to
\[ \sum_{\rho \in \Sigma(1)} \lfloor -\langle \theta, u_\rho \rangle \rfloor D_\rho \]
for some $\theta \in M_\RR$.
We write the Bondal-Thomsen element associated to $\theta$ as $-d_\theta$. 
If $-d \in \Theta$, then $d$ has image lying in at least one chamber $\Gamma_i$ of $\Sigma_{GKZ}(X)$, and we define
\begin{equation} \label{eq:BTdef} \cO_{\Cox} (-d) = \pi_i^* \cO_{\sX_i}(-d) \end{equation}
as in \cite[Definition 4.1]{king}, which is independent of the choice of $\sX_i$ on which $d$ is nef by \cite[Proposition 4.2]{king}.

\section{Extending the equivariant structure to the Cox category}\label{sec:equivariant cover}
As noted in \cref{sec:background}, the definition of the Cox category involves the choice of a smooth toric DM stack $\tsX$ with proper birational maps $\pi_i\colon \tsX \to \sX_i$ that are given by refinements of stacky fans, for all $1\leq i \leq r$.  
To define the $G$-action on the Cox category, it will be helpful to construct an equivariant structure on the secondary fan $\Sigma_{GKZ}$ and construct a cover $\tsX$ that is compatibly equivariant. 
In this section, we show that this can be achieved.  
For convenience, the main result is summarized below, though the proof will involve several distinct steps.

\begin{prop}\label{prop:ExtendingEquivariance}
    We have the following:
    \begin{enumerate}
        \item  $G$ acts on $\Sigma_{GKZ}(X)$.  In particular, each $g\in G$ permutes the maximal chambers $\{\Gamma_1, \dots, \Gamma_r\}$  of $\Sigma_{GKZ}(X)$.  Thus, for each pair $(g,\sX_i)$, the element $g$ induces a toric isomorphism $\phi_g\colon \sX_i \to \sX_{g(i)}$.
        \item  There exists a smooth toric DM stack $\tsX$ with proper birational maps $\pi_i\colon \tsX \to \sX_i$ that are given by refinements of stacky fans, for all $1\leq i \leq r$, and such that the action of $G$ on $N$ induces automorphisms of the stacky fan of $\tsX$.
        \item  These $G$ actions are compatible in the following sense. For each $g\in G$, let $\tilde{\phi}_g \colon \tsX \to \tsX$ be the induced map.  We have commutative diagrams
    \begin{equation} \label{eq:morhpismdiag} \begin{tikzcd}
        \tsX \arrow{r}{\tilde{\phi}_g} \arrow{d}{\pi_i} & \tsX \arrow{d}{\pi_{g(i)}} \\
        \sX_i \arrow{r}{\phi_g} & \sX_{g(i)} 
    \end{tikzcd} \end{equation}
        for all $1 \leq i \leq r$. 
    \end{enumerate}
\end{prop}
We will prove this in several steps.  The proof of the first part is a direct computation involving facts about $\Sigma_{GKZ}(X)$.
\begin{proof}[Proof of Proposition~\ref{prop:ExtendingEquivariance}(1)]
    Recall from \cite[Definition 14.4.2]{CLSToricVarieties} that the cones of $\Sigma_{GKZ}(X)$ are obtained as the images in $\Cl(X)_\RR$ of the sets of the form
    \begin{equation*}
        \wt{\Gamma}_{\overline{\Sigma}, I_\emptyset} = \left\{ a \in \RR^{\Sigma(1)}
        \ \bigg\vert\
        \begin{array}{l} \text{there is a concave\footnotemark} \text{ support function } F \text{ on } \overline{\Sigma} \text{ such that} \\ F(u_\rho) = -a_\rho \text{ for all } \rho \not\in I_\emptyset \text{ and } F(u_\rho) \geq -a_\rho  \text{ for all } \rho \in \Sigma(1)  \end{array} \right\}
    \end{equation*}
    for some $I_\emptyset \subset \Sigma(1)$ and for a generalized fan $\overline{\Sigma}$ in $N_\RR$ whose cones are all generated by $u_\rho$ with $\rho \in \Sigma(1) \setminus I_\emptyset$.\footnotetext{We take the opposite convention of \cite{CLSToricVarieties} and say that the function $-x^2$ is concave.}

    Our first goal is to show that $g$ acts on these cones.
    We note that $G$ acts on $\RR^{\Sigma(1)}$ by $(g \cdot a)_\rho = a_{g^{-1}(\rho)}$. 
    As an automorphism of $N$, $G$ also acts on generalized fans where the cones of $g \cdot \overline{\Sigma}$ are $g(\sigma)$ for $\sigma \in \overline{\Sigma}$.
    In particular, the rays of $g \cdot \overline{\Sigma}$ are generated by $g(\overline{\Sigma}(1))$. 
    Further, if $F$ is a concave support function on $\overline{\Sigma}$ with $F(u_\rho) = - a_\rho$ for all $\rho \not \in I_\emptyset$ and $F(u_p) \geq - a_\rho$, then $F \circ g^{-1}$ is a concave support function function on $\overline{\Sigma}$ such that
    \[ F \circ g^{-1}(u_\rho) = F \circ g^{-1} (g(u_{g^{-1}(\rho)})) \begin{cases}  = -a_{g^{-1}(\rho)} & \text{ for all } \rho \not \in g(I_\emptyset) \\ \geq - a_{g^{-1}(\rho)} & \text{ for all } \rho \in \Sigma(1) \end{cases} . \]
    That is, we have shown that $g(\wt{\Gamma}_{\overline{\Sigma}, I_\emptyset}) = \wt{\Gamma}_{g \cdot \overline{\Sigma}, g(I_\emptyset)} $

    Now, we will show that this action respects the quotient to $\Cl(X)_\RR$ and hence induces an action on $\Sigma_{GKZ}(X)$. If there is an $m \in M_\RR$ such that $a_\rho = \langle m, u_\rho \rangle$ for all $\rho \in \Sigma(1)$, then 
    \[ (g \cdot a)_\rho = a_{g^{-1}(\rho)} = \langle m, u_{g^{-1}(\rho)} \rangle = \langle m, g^{-1}(u_\rho) \rangle = \langle g_* m , u_\rho \rangle \]
    for all $\rho \in \Sigma(1)$ where $g_* \colon M_\RR \to M_\RR$ is the dual to $g^{-1}$. 
    
    To conclude, it remains to check that this action takes chambers to chambers. However, this is immediate from the characterization of chambers in \cite[Proposition 14.4.9(b)]{CLSToricVarieties} and our definition of the action.
    It is also clear that the map of fans given by $g\in G$ is a map of stacky fans as it induces a bijection between the $\Sigma_{\Gamma_i}(1)$ and $\Sigma_{g \cdot \Gamma_i}(1)$.
\end{proof}

\begin{remark}\label{rmk:GactiononSet}
For the remainder of the paper, we will denote the $G$-action on $\{1, \dots, r\}$ coming from the action in \cref{prop:ExtendingEquivariance}(1) on the chambers of $\Sigma_{GKZ}(X)$ by $g\cdot i = g(i)$.
\end{remark}

We break up the proof of Proposition~\ref{prop:ExtendingEquivariance}(2) into two steps:  first, in Lemma~\ref{lem:equivariantRefinement} we use combinatorial methods to produce an equivariant, simplicial refinement of the (non-stacky) fans $\Sigma_i$; then in Lemma~\ref{lem:equivariantRefinementStacky}, we show that this can be upgraded to a refinement of the stacky fans as well.

\begin{lemma}\label{lem:equivariantRefinement}
There is a simplicial fan $\widetilde{\Sigma}$ with the following properties:
\begin{enumerate} 
\item It is a common refinement of all the $\Sigma_i$.
\item The $G$-action on $N$ induces an automorphism of the fan $\widetilde{\Sigma}$.
\end{enumerate}
\end{lemma}
\begin{proof}
We will construct this fan in two steps.  First we take the minimal common refinement $\Sigma_{ref}$ of the $\Sigma_i$.  This is the fan whose cones $\sigma\in \Sigma_{ref}$ are of the form $\sigma= \sigma_1\cap\sigma_2 \cap \cdots \cap \sigma_r$ where $\sigma_i$ is a cone of $\Sigma_i$; see~\cite[Definition 7.6]{ziegler} or \cite[p. 742]{CLSToricVarieties} for details on this construction.   
Since $g\in G$ induces a map of fans $g\colon \Sigma_i\to \Sigma_{g(i)}$, it follows that $g$ sends each cone $\sigma_i$ of $\Sigma_i$ to a cone $g(\sigma_i)$ in $\Sigma_{g(i)}$.  We thus have a natural $G$-action on the cones of $\Sigma_{ref}$ where
\[
g(\sigma) = g(\sigma_1) \cap \cdots \cap g(\sigma_r).
\]

However, $\Sigma_{ref}$ may fail to be simplicial.  So for the second step, we let $\widetilde{\Sigma}$ be the barycentric subdivision of $\Sigma_{ref}$, with barycenters chosen to be the sums of primitive generators.
Specifically, for each cone $\sigma \in \Sigma_{ref}$, let $\rho_\sigma$ be the ray generated (over the nonnegative rational numbers) by $\sum_{\rho \in \sigma} u_\rho$. Then
\[
\widetilde{\Sigma} = \{ \{\rho_{\sigma_1}, \rho_{\sigma_2}, \dots, \rho_{\sigma_k}\} : \{0\} \subsetneq \sigma_1 \subsetneq \sigma_2 \subsetneq \dots \subsetneq \sigma_k \in \Sigma_{ref} \}.
\]

Finally, since $G$ naturally acts on cones of $\Sigma_{ref}$ and preserves inclusions, $G$ also naturally acts on chains of cones.   This induces the $G$ action on $\widetilde{\Sigma}$ as desired.
\end{proof}

Next we turn to the stacky structure.  Instead of considering only the fans $\Sigma_i$, we now consider the pairs $(\Sigma_i, \beta_i)$ where each $\beta_i$ is the map $\ZZ^{\Sigma_i(1)}\to N$ that sends each basis element $e_\rho$ to the primitive point $u_\rho$ on the corresponding ray $\rho$. 

\begin{lemma}\label{lem:equivariantRefinementStacky}
With notation as above, we can extend the simplicial fan $\widetilde{\Sigma}$ to a stacky fan $(\widetilde{\Sigma}, \widetilde{\beta})$ with the following properties.
\begin{enumerate} 
\item The identity on $N$ lifts to a map of stacky fans $(\widetilde{\Sigma}, \widetilde{\beta}) \to (\Sigma_i,\beta_i)$ for all $1\leq i \leq r$.
\item The $G$-action on $N$ is contained in the automorphism group of the stacky fan $(\widetilde{\Sigma}, \widetilde{\beta})$.
\end{enumerate}
\end{lemma}
\begin{proof}
We will define $\widetilde{\beta}$ by determining $b_\rho \in \ZZ_{>0}$ such that $\widetilde{\beta}(e_\rho) = b_\rho u_\rho$ for all $\rho \in \widetilde{\Sigma}(1)$. 

The construction of a $\widetilde{\beta}$ satisfying the first property is carried out in \cite[Section 3]{king}.
We will now recall this construction and give general constraints on the $b_\rho$ before addressing the $G$ action.
For each $i$ and $\rho \in \widetilde{\Sigma}(1)$, there is a unique minimal cone $\sigma_{\rho, i} \in \Sigma_i$ containing $\rho$.
As $\Sigma_i$ is simplicial, we can write 
\[ u_{\rho} = \sum_{\tau \in \sigma_{\rho, i}(1)} c_\tau u_\tau\]
where $c_\tau \in \QQ_{>0}$. Clearing denominators, we obtain
\[ a_{\rho,i} u_\rho = \sum_{\tau \in \sigma_{\rho, i}(1)} a_\tau u_\tau\]
where $a_{\rho, i}, a_\tau \in \ZZ_{>0}$ and $\gcd(a_{\rho,i}, a_\tau) = 1$ for all $\tau$.  
We conclude that there is a map of stacky fans $(\widetilde{\Sigma}, \widetilde{\beta}) \to (\Sigma_i,\beta_i)$ as long as $b_\rho$ is a multiple of $a_{\rho, i}$ for all $\rho \in \widetilde{\Sigma}(1)$. 
Thus, we satisfy the first property if (and only if) $b_\rho$ is a multiple of $\mathrm{lcm}_{1 \leq i \leq r}(a_{\rho,i})$ for all $\rho \in \widetilde{\Sigma}(1)$. 

Now, we turn to the $G$-action. Since every $g \in G$ is an automorphism of $\widetilde{\Sigma}$, we have that $g$ sends every ray $\rho \in \widetilde{\Sigma}(1)$ to another ray $g(\rho)$. 
Thus, $g$ induces a map of stacky fans if and only if $b_\rho$ is a multiple of $b_{g(\rho)}$ for all $g \in G$.

Since $G$ is finite, it is possible to satisfy all these conditions. For instance, we can set
\[ b_\rho = \prod_{g \in G} \underset{1 \leq i \leq r}{\mathrm{lcm}} (a_{g(\rho), i}) \]
for all $\rho \in \widetilde{\Sigma}(1)$ (or we can again take a least common multiple over $G$ if one wishes to minimize the $b_\rho$).
\end{proof}

\begin{proof}[Proof of Proposition~\ref{prop:ExtendingEquivariance}(2)]
    This follows immediately from Lemmas~\ref{lem:equivariantRefinement} and \ref{lem:equivariantRefinementStacky}.
\end{proof}

Finally, we observe that the equivariant structures are compatible.
\begin{proof}[Proof of Proposition~\ref{prop:ExtendingEquivariance}(3)]
Referring to \eqref{eq:morhpismdiag}:  the underlying maps of lattices commute, since the vertical arrows correspond to the identity and the horizontal arrows both correspond to $g$.   Thus the corresponding toric morphisms agree as well.
 \end{proof}

\section{Proof of Theorem~\ref{thm:main}}\label{sec:proofOfMain}

In this section, we prove \cref{thm:main} after some preparatory lemmas. 
When $\tsX$ satisfies the conclusions of \cref{prop:ExtendingEquivariance}, we will write $G \subseteq \mathrm{Aut}(\tsX)$ as a slight abuse of notation. 

\begin{lemma}\label{lem:GactCox}
    If $G \subseteq \mathrm{Aut}(\tsX)$, then $D_{\Cox}(X)$ is closed under the induced action of $G$ on $D(\tsX)$. Moreover, the induced action\footnote{This is a strict action in the sense of \cite{shinderaction} as automorphisms act strictly on the derived category.} on $D_{\Cox}(X)$ is independent of the choice of $\tsX$ with $G \subseteq \mathrm{Aut}(\tsX)$ and is given by 
    \begin{equation} \label{eq:generatorsaction}
    g\cdot \pi_i^* \cE = \pi_{g(i)}^* \phi_{g\ast}\cE.
    \end{equation}
    for $\cE \in D(\sX_i)$.
\end{lemma}
\begin{proof}
   By Proposition~\ref{prop:ExtendingEquivariance}(3), we have that $\tilde{\phi}_{g \ast} \pi_i^* D(\sX_i) \subseteq \pi_{g(i)}^* D(\sX_{g(i)})$.
    In particular, this implies that the subcategory of $D(\tsX)$ consisting of all the $\pi_i^*D(\sX_i)$ is closed under the $G$-action. Since the $\tilde{\phi}_{g\ast}$ are triangulated functors, the triangulated hull $D_{\Cox}(X)$ is closed under the $G$-action as well.
    The formula \eqref{eq:generatorsaction} for the action on pullbacks also follows immediately from commutativity of \eqref{eq:morhpismdiag}.
\end{proof}

\begin{lemma}\label{lem:GTheta}
    $\Theta$ is invariant under the $G$-action on $D_{\Cox}(X)$.
\end{lemma}
\begin{proof}
    We will show that $g \cdot \cO_{\Cox} (-d_\theta) = \cO_{\Cox} (-d_{g_*\theta})$ for any $\theta \in M_\RR$ where $g_* \in \mathrm{Aut}(M)$ is the dual to $g^{-1}$. 
    By \cref{lem:GactCox}, we have that $g \cdot \cO_{\Cox} (-d_\theta) = \pi_{g(i)}^* \phi_{g \ast}\cO_{\sX_i}(-d_\theta)$ where $d_\theta$ has image in $\Gamma_i$. 
    Now, we note that
    \[ \phi_{g\ast}\cO_{\sX_i}(-d_\theta) = (\phi_g^{-1})^* \cO_{\sX_i}(-d_\theta). \]
    Thus, $\phi_{g\ast}\cO_{\sX_i}(-d_\theta)$ is the line bundle on $\sX_{g(i)}$ corresponding to the toric divisor $\sum e_\rho D_\rho$ where
    \[ e_\rho = (-d_\theta)_{g^{-1}(\rho)} = \lfloor - \langle \theta, g^{-1}(u_\rho) \rangle \rfloor = \lfloor - \langle g_* \theta, u_\rho \rangle \rfloor \]
    for all $\rho \in \Sigma_{g(i)}(1)$. 
    That is, $\phi_{g\ast} \cO_{\sX_i}(-d_\theta) = \cO_{\sX_{g(i)}} (-d_{g_* \theta})$. 
    
    It remains only to show that $\cO_{\Cox}(-d_{g_* \theta}) = \pi_{g(i)}^*\cO_{\sX_{g(i)}}(-d_{g_* \theta})$, that is, $d_{g_*\theta}$ is nef on $\sX_{g(i)}$. To see this, observe that the support function for $d_{g_* \theta}$ is given by $F \circ g^{-1}$ where $F$ is the support function for $d_\theta$. In particular, $F \circ g^{-1}$ is a concave support function on $\Sigma_{g(i)}$ as $F$ is a concave support function on $\Sigma_i$ and $g^{-1}$ is linear.
\end{proof}

We can now prove the following generalization of Theorem~\ref{thm:main}:
\begin{thm}\label{thm:generalMain}
    Let $X$ be a semiprojective normal toric variety.  The Cox category $D_{\Cox}(X)$ of $X$ has a $G$-invariant tilting bundle $\oplus_{-d\in \Theta} \cO_{\Cox}(-d)$.
    
    If $X$ is projective, then $\Theta$ forms a $G$-invariant full strong exceptional collection with respect to any ordering that refines the effective partial order.
\end{thm}
\begin{proof}
    In the semiprojective case, \cite[Theorem A]{king} shows that $\oplus_{-d\in \Theta} \cO_{\Cox}(-d)$ is a tilting bundle for $D_{\Cox}(X)$ and \cref{lem:GTheta} shows that it is $G$-invariant.   In the projective case, the statement also follows from \cite[Theorem A]{king} and \cref{lem:GTheta}.
\end{proof}

The following implies \cref{cor:summand}.
\begin{cor} \label{cor:generalcorB}
    Let $X$ be a semiprojective, simplicial toric variety with associated stack $\sX$. The Grothendieck group $K_0(\sX)$ is a summand of the permutation $G$-module $K_{\Cox}(X)$.
\end{cor}
\begin{proof}
    Write $\pi^*\colon D(\sX)\to D_{\Cox}(X)$ and $\pi_*\colon D_{\Cox}(X)\to D(\sX)$.  We have that $\pi^*$ is a fully faithful embedding, with adjoint $\pi_*$.  In particular, the induced maps on $K_0$ realize $K_0(\sX)$ as a split summand of $K_0(D_{Cox}(X))$. 
\end{proof}

\section{Further invariant structures}\label{sec:furtherInvariant}

In this section, we will examine some further structures preserved by the $G$-action that are a consequence of viewing the Bondal-Thomsen elements as corresponding to subsets of the torus $M_\RR/M$. 
Namely, we saw in the proof of \cref{lem:GTheta} that $g \in G$ acts on the Cox category by taking $\cO_{\Cox}(-d_\theta)$ to $\cO_{\Cox}(-d_{g_* \theta})$ where $g_* \in \Aut(M)$ is the dual to $g^{-1}$.
We can extend this as follows.
Suppose we have $-d \in \Theta$ and we choose $\theta \in M_\RR$ such that $-d = -d_\theta$. 
Then, we have a lift of $-d_\theta$ to $\ZZ^{\Sigma(1)}$ given by \eqref{eq:BTdef}. 
We then obtain the set 
\[ S_\theta = \left\{ k \in M_\RR \mid \lceil \langle \theta, u_\rho \rangle \rceil - 1 < \langle k, u_\rho \rangle \leq \lceil \langle \theta, u_\rho \rangle \rceil \text{ for all } \rho \in \Sigma(1) \right\}  \]
whose projection to $M_\RR/M$ is the Bondal-Thomsen stratum of $-d_\theta$, that is, the set of points in $M_\RR/M$ that correspond $-d_\theta$ under the assignment \eqref{eq:BTdef}. 
The Bondal-Thomsen strata are the strata of the oriented Bondal stratification which is given by the dual hyperplanes to $\Sigma(1)$.
We also consider the polytope
\[ P_\theta = \left\{ k \in M_\RR \mid \langle k, u_\rho \rangle \leq \lceil \langle \theta, u_\rho \rangle \rceil \text{ for all } \rho \in \Sigma(1) \right\}. \]
Note that $P_\theta = - P_{d_\theta}$ where $d_\theta$ is the section polytope of the effective divisor $d_\theta$.
Also, observe that $S_\theta \subseteq P_\theta$ is relatively open. 
The polytopes of the form $P_\theta$ govern the morphisms in $D_{\Cox}(X)$ by \cite[Corollary 4.25]{king}. 

Clearly, $g_*$ acts on the sets of $S_\theta$ and $P_\theta$, which can be useful in understanding and visualizing the orbits of the $G$ action on $\Theta$. In fact, since $g_* \in \Aut(M)$, $g_*$ acts on the equivalence classes of these strata and polytopes up to translation by $M$ and descends to an action on the torus $M_\RR/M$. 

We further set 
\[ J_\theta = \left\{ \rho \in \Sigma(1) \mid \langle k , u_\rho \rangle = \lceil \langle \theta, u_\rho \rangle \rceil \text{ for all } k \in S_\theta \right\}  \]
Note that $J_\theta$ depends only on $S_\theta$ up to translation by $M$ and that one could replace $S_\theta$ with $P_\theta$ in the definition of $J_\theta$ if desired as $P_\theta$ is connected.

\begin{example} \label{ex:P2}
    Suppose that $\Sigma(1) = \{ e_1, e_2, -e_1 - e_2 \}$ so that we are studying $D_{\Cox}(\PP^2) = D(\PP^2)$.
    In this case, $G$ is isomorphic to $S_3$ where the action is given by the toric automorphisms permuting the homogeneous coordinates on $\PP^2$. 
    The Bondal-Thomsen collection coincides with the Beilinson collection $\cO(-2), \cO(-1), \cO$.
    In this example, the $G$ action on $\Theta$ fixes all three of these line bundles.
    Nevertheless, we will take advantage of this example to illustrate the definitions given thus far in this section.
    
    Namely, $\cO$ corresponds to $\theta_0 = (0,0)$. We have $S_{\theta_0} = P_{\theta_0} = \{ (0,0) \}$ and $J_{\theta_0} = \Sigma(1)$.
    In addition, $\cO(-1)$ corresponds to $\theta_1 = (-1/3, -1/3)$ among many other possible choices.
    We then have $S_{\theta_1} = P_{\theta_1} = \{ (k_1, k_2) \mid k_1, k_2 \leq 0 \text{ and } k_1 + k_2 \geq -1 \}$ and $J_\theta = \emptyset$. 
    Similarly, $\cO(-2)$ corresponds to $\theta_2 = (-2/3, -2/3)$, for example. In this case, we again have $J_\theta =\emptyset$ while $S_{\theta_2} = \{ (k_1, k_2) \mid k_1, k_2 > -1 \text{ and } k_1 + k_2 < -1\}$ differs from $P_{\theta_2} = \{ (k_1, k_2) \mid k_1, k_2 \leq 0 \text{ and } k_1 + k_2 \geq -2 \}$.
    See \Cref{fig:stratanadpolytopesP2}.
\end{example}

\begin{figure}
    \begin{center}
    \begin{tikzpicture}[decoration=border]
    \begin{scope}[scale=3]

    \draw [thick, postaction={draw,decorate, black, decoration={border, amplitude=0.09cm, angle=90, segment length = .25cm}}]
         (0,2) -- (0,0) -- (2,0);
    \draw [thick, postaction={draw,decorate, black, decoration={border, amplitude=0.09cm, angle=90, segment length = .25cm}}]
         (0,2) -- (2,2) -- (2,0);
    \draw [thick, postaction={draw,decorate, black, decoration={border, amplitude=0.09cm, angle=90, segment length = .25cm}}]
        (0,1) -- (2,1);
    \draw [thick, postaction={draw,decorate, black, decoration={border, amplitude=0.09cm, angle=90, segment length = .25cm}}]
        (1,2) -- (1,0);
    \draw [thick, postaction={draw,decorate, black, decoration={border, amplitude=0.09cm, angle=90, segment length = .25cm}}]
        (2,0) -- (0,2);
    \draw [thick, postaction={draw,decorate, black, decoration={border, amplitude=0.09cm, angle=90, segment length = .25cm}}]
        (1,0) -- (0,1);
    \draw [thick, postaction={draw,decorate, black, decoration={border, amplitude=0.09cm, angle=90, segment length = .25cm}}]
        (2,1) -- (1,2);

    \node[label=right:{$\theta_1$}] at (1.66,1.66) {$\bullet$};
    \draw[ultra thick, blue, fill, fill opacity= 0.2] 
        (2,2) -- (1,2) -- (2,1) -- (2,2);
    \node[label=above:{\color{blue} $S_{\theta_1} = P_{\theta_1}$}] at (1.5, 2) { };

    \begin{scope}[shift = {(3,0)}]
    \draw [thick, postaction={draw,decorate, black, decoration={border, amplitude=0.09cm, angle=90, segment length = .25cm}}]
         (0,2) -- (0,0) -- (2,0);
    \draw [thick, postaction={draw,decorate, black, decoration={border, amplitude=0.09cm, angle=90, segment length = .25cm}}]
         (0,2) -- (2,2) -- (2,0);
    \draw [thick, postaction={draw,decorate, black, decoration={border, amplitude=0.09cm, angle=90, segment length = .25cm}}]
        (0,1) -- (2,1);
    \draw [thick, postaction={draw,decorate, black, decoration={border, amplitude=0.09cm, angle=90, segment length = .25cm}}]
        (1,2) -- (1,0);
    \draw [thick, postaction={draw,decorate, black, decoration={border, amplitude=0.09cm, angle=90, segment length = .25cm}}]
        (2,0) -- (0,2);
    \draw [thick, postaction={draw,decorate, black, decoration={border, amplitude=0.09cm, angle=90, segment length = .25cm}}]
        (1,0) -- (0,1);
    \draw [thick, postaction={draw,decorate, black, decoration={border, amplitude=0.09cm, angle=90, segment length = .25cm}}]
        (2,1) -- (1,2);

    \draw[ultra thick, blue, fill, fill opacity= 0.2] 
        (0,2) -- (2,0) -- (2,2) -- (0,2);
    \draw[ultra thick, red, dashed, fill, fill opacity= 0.4] 
        (1,1) -- (1,2) -- (2,1) -- (1,1);
    \node[label=left:{$\theta_2$}] at (1.33,1.33) {$\bullet$};
    \node[label=above:{\color{blue} $P_{\theta_2}$}] at (1.5, 2) { };
    \node[label=right:{\color{red} $S_{\theta_2}$}] at (1.5, 1.5) { };
    \end{scope}
    
    \end{scope}
    \end{tikzpicture}
    \end{center}
    \caption{Examples of $S_\theta$ and $P_\theta$ drawn on the periodic hyperplane arrangement on $M_\RR$ that lifts the Bondal stratification for $\PP^2$ with $\theta_1$ and $\theta_2$ corresponding to $\cO(-1)$ and $\cO(-2)$, respectively.}
    \label{fig:stratanadpolytopesP2}
\end{figure}
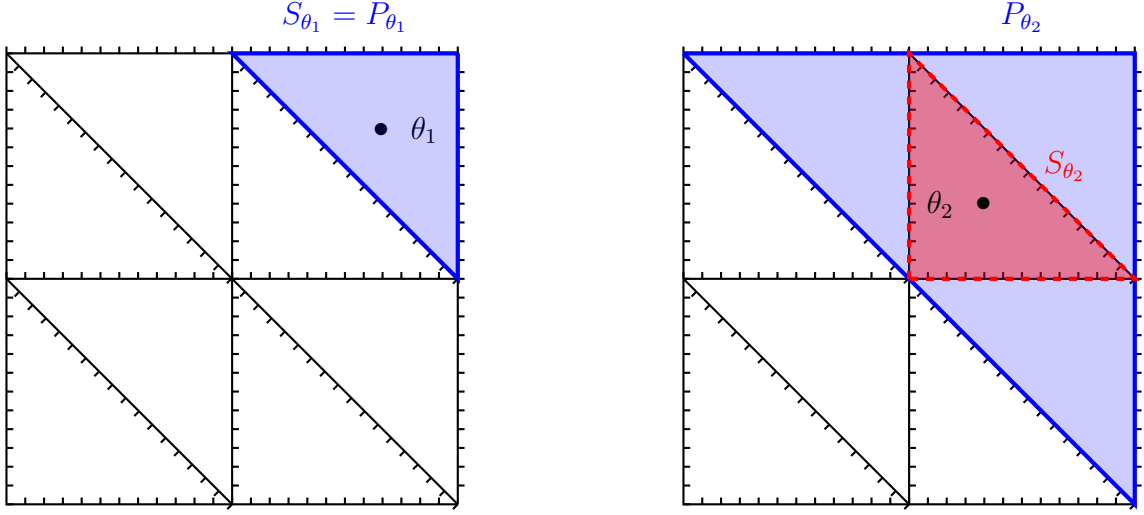

The following statement is hinted at in \cite{bauermeister2025strata} where the dimenson of $S_\theta$ (or equivalently the dimension of $P_\theta$) is related to the dimension of the minimal face of the zonotope in $\Cl(X)$ containing the image of $d_\theta$ and the size of $J_\theta$.

\begin{lemma} \label{lem:orthoJ}
    If 
    \begin{equation} \label{eq:nonzerohom}
    \Hom_{D_{\Cox}(X)} (\cO_{Cox}(-d_\theta), \cO_{Cox}(-d_{\theta'})) \neq 0, 
    \end{equation}
    then $J_{\theta} \subseteq J_{\theta'}$.
\end{lemma}
\begin{proof}
    By \cite[Corollary 4.25]{king} (and mulitiplying by $-1$), \eqref{eq:nonzerohom} implies that there is a lattice point $m \in M$ such that $m + \theta' \in P_\theta$. Now, suppose that $\rho \in J_{\theta}$ so that $\langle k , u_\rho \rangle = \lceil \langle \theta, u_\rho \rangle \rceil$ for all $k \in P_\theta$. 
    In particular, we obtain that
    \[ \langle \theta', u_\rho \rangle = \lceil \langle \theta, u_\rho \rangle \rceil + \langle m, u_\rho \rangle \in \ZZ. \]
    As we can replace $\theta'$ with any $k \in S_{\theta'}$, we obtain that $\langle k, u_\rho \rangle \in \ZZ$ for all $k \in S_{\theta'}$ and hence $\rho \in J_{\theta'}$ since $S_\theta$ is connected. 
\end{proof}

Now, we can use \cref{lem:orthoJ} to make a couple further observations about the Cox category. 

\begin{prop} \label{prop:dimSOD}
    Assume $X$ is projective. 
    For $0 \leq \delta \leq n = \dim(X)$, let $\cA_\delta$ be the subcategory of $D_{\Cox}(X)$ generated by the $\cO_{\Cox}(-d_\theta)$ with $\dim(S_\theta) = \delta$. Similarly, for $0 \leq j \leq N = |\Sigma(1)|$, let $\cB_j$ be the subcategory of $D_{\Cox}(X)$ generated by the $\cO_{\Cox}(-d_\theta)$ with $|J_\theta| = j$. Then, there is a dimension of strata semi-orthogonal decomposition
    \begin{equation} \label{eq:dimSOD}
    D_{Cox}(X) = \langle \cA_n, \hdots, \cA_0 \rangle
    \end{equation}
    and a depth of strata semi-orthogonal decomposition
    \begin{equation} \label{eq:jSOD}
    D_{Cox}(X) = \langle \cB_0, \hdots, \cB_{N} \rangle.
    \end{equation}
\end{prop}
\begin{proof}
    Since the $\cO_{Cox}(-d_\theta)$ are all exceptional objects, the $\cA_\delta$ and $\cB_j$ are all admissible subcategories. It therefore is enough to show that $\cA_\delta$ is in the left orthogonal to $\cA_{\delta'}$ if $\delta < \delta'$ and that $\cB_j$ is in the left orthogonal to $\cB_{j'}$ if $j' < j$.
    The latter immediately follows from \cref{lem:orthoJ} since $|J_{\theta'}| < |J_{\theta}|$ implies that $J_\theta \not \subseteq J_{\theta'}$. 
    For the former, we also apply \cref{lem:orthoJ} and note that $J_\theta \subseteq J_{\theta'}$ implies that $\dim(S_\theta) \geq \dim(S_{\theta'})$.  
\end{proof}

As the $G$-action is reflected by its action on $M_\RR/M$, we obtain that it preserves these semi-orthogonal decompositions.

\begin{prop} \label{prop:Gsod}
    If $X$ is a projective toric variety, the $G$ action on $D_{Cox}(X)$ preserves the dimension and depth of strata semi-orthogonal decompositions.
\end{prop}
\begin{proof}
    For both semi-orthogonal decompositions \eqref{eq:dimSOD} and \eqref{eq:jSOD}, we observe that these are determined by $S_\theta$. The claim then follows from the fact that $g_* S_\theta = S_{g_*\theta}$ and that the proof of \cref{lem:GTheta} shows that $g \cO_{\Cox}(-d_\theta) = \cO_{\Cox}(-d_{g_*\theta})$.
\end{proof}

In particular, each orbit of the $G$ action on $\Theta$ is contained in one of the $\cA_\delta$ and one of the $\cB_j$.
In \cref{ex:P2}, the two semi-orthogonal decompositions coincide, but the orbits of $G$ are smaller.
Typically, the two semi-orthogonal decompositions and the orbits all differ.
We will see further examples of this behavior in \cref{sec:permex} which examines permutohedral varieties.

We will conclude this section with an observation in another direction that also takes advantage of interpeting the $G$-action on $\Theta$ in terms of $M_\RR/M$. 
One of the utilities of the Cox category is as a repository for uniform resolutions by the line bundles in $\Theta$ across all of the $\sX_i$. 
For instance, it is a natural home to the resolutions constructed in \cite{HHL}. 
We briefly recall the construction of these HHL resolutions from \cite{HHL} (see also \cite{BCHSY,short-HST,borisov2025hanlon}).
In fact, we will describe the lifts of these resolutions to the Cox category (by \cite{BCHSY, borisov2025hanlon}, the HHL resolutions are in fact resolutions in the homotopy category of free $S$-modules supported in $\Theta$, a category equivalent to $D_{\Cox}(X)$).
There is an HHL resolution associated to any lattice inclusion $\phi \colon N' \hookrightarrow N$. 
We set $T_\phi$ to be the kernel of the induced map $M_\RR/M \to M_\RR'/M'$ and $\widetilde{T}_\phi$ to be its preimage in $M_\RR$. 
Each $S_\theta \cap \widetilde{T}_\phi$ has closure a polytope in $\widetilde{T}_\phi$. 
The degree $j$ part of the HHL resolution is the direct sum of $\cO_X (-d_\theta)$ over the images in $M_\RR/M$ of faces of the $S_\theta$ of dimension $j$. 
If $-\theta'$ (possibly equal to $-\theta$) lies on a face of the closure of $S_\theta$, there is a morphism $\cO_X(-d_\theta) \to \cO_X(-d_{\theta'})$ corresponding to $-\theta' \in P_\theta$.
After choosing orientations for all the images of faces of $S_\theta$ on $M_\RR/M$, these morphisms are the maps in the HHL resolution with a sign determined by whether or not the induced orientation agrees with the chosen one on a face.
Note that different choices of orientations give isomorphic complexes.
The main result of \cite{HHL} implies that after pushing forward to $\sX_i$, this construction resolves the structure sheaf of the normalization of the closure of $N' \otimes \mathbb{G}_m$ in $\sX_i$ for all $i$. 
From the observations in this paper, we can deduce certain symmetries of this construction.

\begin{prop} \label{prop:ginvHHL}
    Suppose $X$ is a semiprojective toric variety. If the image of $\phi$ is $\Gamma$-invariant for a subgroup $\Gamma \subseteq G$, then the HHL resolution associated to $\phi$ is $\Gamma$-invariant up to isomorphism.
\end{prop}
\begin{proof}
    If the image of $\phi$ is $\Gamma$-invariant, then $T^\phi$ and $\widetilde{T}^\phi$ are also $\Gamma$-invariant under the dual action. 
    The proof of \cref{lem:GTheta} shows that $g \cdot \cO_{\Cox}(-d_\theta)$ is consistent with action of $G$ on the sets $S_\theta$ as in the proof of \cref{prop:Gsod}. 
    Further, the action of $(\phi_g)_*$ on morphisms is determined by the underlying action on lattices which takes $-\theta' \in P_\theta$ to $-g_*\theta' \in P_{g_* \theta}$.
    As each $g \in G$ is not guaranteed to preserve our chosen orientations, we conclude that $\Gamma$ preserves the HHL resolution associated to $\phi$ up to sign, and hence up to isomorphism.
\end{proof}

Note that $0 \in N$ is always $G$-invariant and thus the HHL resolution of the structure sheaf of the identity point is always invariant up to isomorphism under the $G$ action. 
Similarly, the diagonal in $N \times N$ is always invariant under the diagonal subgroup of $G \times G$, and hence \cref{prop:ginvHHL} implies that the HHL resolution of the diagonal is invariant up to isomorphism under the action of the diagonal subgroup of $G \times G$.

\section{Example: permutohedral variety} \label{sec:permex}

We now explore the $G$-action on the Cox category of a permutohedral variety with the aid of some of the observations of \cref{sec:furtherInvariant}.
First, recall that the permutohedral variety $X_{n+1}$ with $n \geq 2$ is the toric variety whose moment polytope is the permutohedron.
The fan of $X_{n+1}$, sometimes called the braid fan, is most symmetrically presented as lying in $\RR^{n+1}/\langle (1,\hdots, 1) \rangle$ with rays generated by the images of the vectors $e_S$ where $\emptyset \neq S \subsetneq \{1, \hdots, n +1 \}$ and $e_S = \sum_{i \in S} e_i$. 
The group $G$ of symmetries of these vectors is isomorphic to $S_{n+1} \times S_2$, which is also the toric automorphism group of $X_{n+1}$. 
The first factor of $G \cong S_{n+1} \times S_2$ acts by permuting the elements of $S$, and the second factor acts by taking $S$ to its complement. 

Less elegantly, $X_{n+1}$ is obtained by blowing up $\PP^{n}$ at all of its toric orbit closures in order of increasing dimension of the blow-up locus, and we can obtain a fan from the usual fan for $\PP^{n}$ in $\RR^n$.
This amounts to setting $e_{n+1} = -e_1 -\hdots - e_n$ to identify the quotient $\RR^{n+1}/\langle (1,\hdots, 1)\rangle$ in our first description of the braid fan with $\RR^n$.

Let us first discuss the case where $n=2$ so we are considering $X_3$ which is the del Pezzo surface obtained by blowing up $\PP^2$ at its three torus fixed points. 
We can choose a basis so that the rays $\Sigma_3(1)$ have generators
\[ \{ \pm e_1, \pm e_2, \pm (e_1 + e_2) \}. \]
The $S_3$ factor of $G$ permutes $e_1$, $e_2$, and $-e_1 - e_2$, and the $S_2$ factor is generated by $-\mathrm{id}$.
In this case, $D_{\Cox}(X_3) = D(X_3)$ and there are six Bondal-Thomsen elements in the full strong exceptional collection. 
These are depicted in \Cref{fig:perm3} on the Bondal stratification using the divisors $H, E_{12}, E_{10}, E_{20}$ as a basis for the Picard group where $H$ is the pullback of the hyperplane class from $\PP^2$ and $E_{ij}$ is the exceptional divisor over the torus fixed point $z_i = z_j = 0$ in homogeneous coordinates. 
The regions $S_\theta$ are the relative interiors of the strata of the triangulation of the square in \Cref{fig:perm3} and the polytopes $P_\theta$ are their closures.
The dimension and depth semi-orthogonal decompositions from \cref{prop:dimSOD} coincide in this example, and it is straightforward to check that theses are also the orbits of the $S_3 \times S_2$ action.
For example, the $S_2$ factor simply permutes $\cO(-H)$ and $\cO(-2H + E_{01} + E_{02} + E_{12})$ while the $S_3$ factor permutes only the three Bondal-Thomsen elements lying on edges of the triangulation in \Cref{fig:perm3}.

\begin{figure}
    \begin{center}
    \begin{tikzpicture}[decoration=border]
    \begin{scope}[scale=2]
    
    \draw [thick, postaction={draw,decorate, black, decoration={border, amplitude=0.09cm, angle=90, segment length = .25cm}}]
        (0,3) -- (3,0);
    \draw [thick, postaction={draw,decorate, black, decoration={border, amplitude=0.09cm, angle=90, segment length = .25cm}}]
        (3,0) -- (0,3);
    \draw [thick, postaction={draw,decorate, black, decoration={border, amplitude=0.09cm, angle=90, segment length = .25cm}}]
        (0,3) -- (0,0);
    \draw [thick, postaction={draw,decorate, black, decoration={border, amplitude=0.09cm, angle=90, segment length = .25cm}}]
        (0,0) -- (0,3);
    \draw [thick, postaction={draw,decorate, black, decoration={border, amplitude=0.09cm, angle=90, segment length = .25cm}}]
        (3,3) -- (3,0);
    \draw [thick, postaction={draw,decorate, black, decoration={border, amplitude=0.09cm, angle=90, segment length = .25cm}}]
        (3,0) -- (3,3);
    \draw [thick, postaction={draw,decorate, black, decoration={border, amplitude=0.09cm, angle=90, segment length = .25cm}}]
        (0,0) -- (3,0) -- (0,0);
    \draw [thick, postaction={draw,decorate, black, decoration={border, amplitude=0.09cm, angle=90, segment length = .25cm}}]
        (0,3) -- (3,3) -- (0,3);

    \node[label=below:{\footnotesize $\cO$}] at (0,0) {$\bullet$};
    \node[label=below:{\footnotesize $\cO(-H + E_{01})$}] at (1.5,0) {$\bullet$};
    \node[label=left:{\footnotesize $\cO(-H + E_{02})$}] at (0,1.5) {$\bullet$};
    \node[label=right:{\footnotesize $\cO(-H + E_{12})$}] at (1.5,1.5) {$\bullet$};
    \node[label=below:{\footnotesize $\cO(-2H + E_{01} + E_{02} + E_{12})$}] at (1.15,0.9) {$\bullet$};
    \node[label=above:{\footnotesize $\cO(-H)$}] at (1.85,2.15) {$\bullet$};
        
    \end{scope}
    \end{tikzpicture}
    \end{center}
    \caption{A fundamental domain of the Bondal stratification for the permutohedral variety $X_3$ with the strata labeled by the corresponding Bondal-Thomsen elements.}
    \label{fig:perm3}
\end{figure}
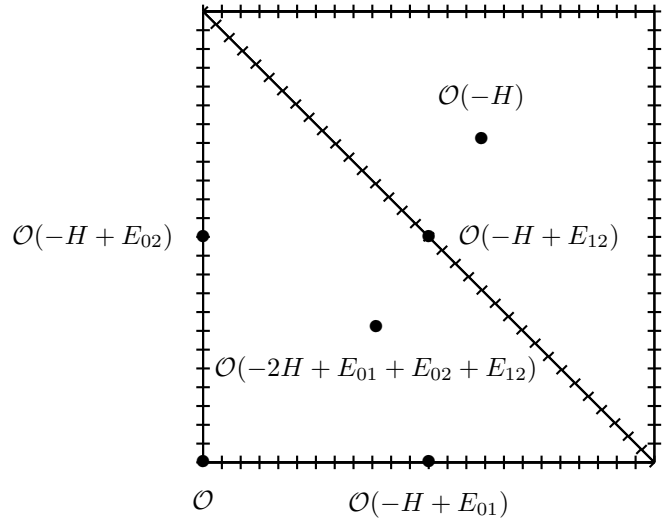

Let's turn our attention to $n = 3$. 
As in the previous dimension, every $S_\theta$ is equal to its relative interior as can be easily checked to be the case for any set of ray generators closed under multiplication by $-1$. 
Now, the Bondal stratification on a fundamental domain slices the cube $[0,1]^3$ into two standard simplices and an octahedron by the hyperplanes $x_1 + x_2 + x_3 = 1$ and $x_1 + x_2 + x_3 = 2$. 
The hyperplanes $x_i + x_j = 1$ further subdivide the octahedron into eight simplices of volume $1/12$. 
See \Cref{fig:perm4}.
From that, we see that there are $2$ vertices, $12$ edges, $20$ faces, and $10$ simplices in the Bondal stratification of the torus leading to a total of $44$ elements in the Bondal-Thomsen collection. 

\begin{figure}
    \begin{center}
\begin{tikzpicture}
\begin{scope}[scale = 4]
\coordinate (O) at (0,0,0);
\coordinate (I) at (1,1,1);

\draw[thick, black] (O) -- (1,0,0) -- (1,0,1) -- (0,0,1) -- cycle;
\draw[thick, black] (O) -- (0,1,0) -- (1,1,0) -- (1,0,0) -- cycle;
\draw[thick, black] (O) -- (0,0,1) -- (0,1,1) -- (0,1,0) -- cycle;
\draw[thick, black,fill=blue, fill opacity = 0.1] (I) -- (1,1,0) -- (1,0,0) -- (1,0,1) -- cycle;
\draw[thick, black,fill=blue, fill opacity = 0.1] (I) -- (1,0,1) -- (0,0,1) -- (0,1,1) -- cycle;
\draw[thick, black,fill=blue, fill opacity = 0.1] (I) -- (1,1,0) -- (0,1,0) -- (0,1,1) -- cycle;
\draw[thick, black] (1,0,0) -- (0,1,0) -- (0,0,1) -- cycle;
\draw[thick, black ] (1,1,0) -- (1,0,1) -- (0,1,1) -- cycle;

\begin{scope}[shift = {(2,0,0)}]
\draw[thick, black,fill=blue, fill opacity = 0.1] (1,1,0) -- (1,0,1) -- (0,1,1) -- cycle;
\draw[thick, black,fill=blue, fill opacity = 0.1] (1,1,0) -- (0,1,1) -- (0,1,0) -- cycle;
\draw[thick, black,fill=blue, fill opacity = 0.1] (1,1,0) -- (1,0,0) -- (1,0,1) -- cycle;
\draw[thick, black] (0,1,0) -- (1,0,0);

\draw[thick, black,fill=blue, fill opacity = 0.1] (1,0,1) -- (0,1,1) -- (0,0,1) -- cycle;
\draw[thick, black] (0,0,1) -- (0,1,0);
\draw[thick, black] (0,0,1) -- (1,0,0);
\draw[thick, black] (1,0,0) -- (0,1,1);
\draw[thick, black] (0,1,0) -- (1,0,1);
\draw[thick, black] (0,0,1) -- (1,1,0);
\end{scope}
\end{scope}
\end{tikzpicture}
\caption{On the left, we cut the cube into two simplices and an octahedron as a first step in understanding the Bondal stratification for $X_4$. On the right, we finish drawing the stratification by cutting the octahedron into eight simplices.}
\label{fig:perm4}
\end{center}
\end{figure}
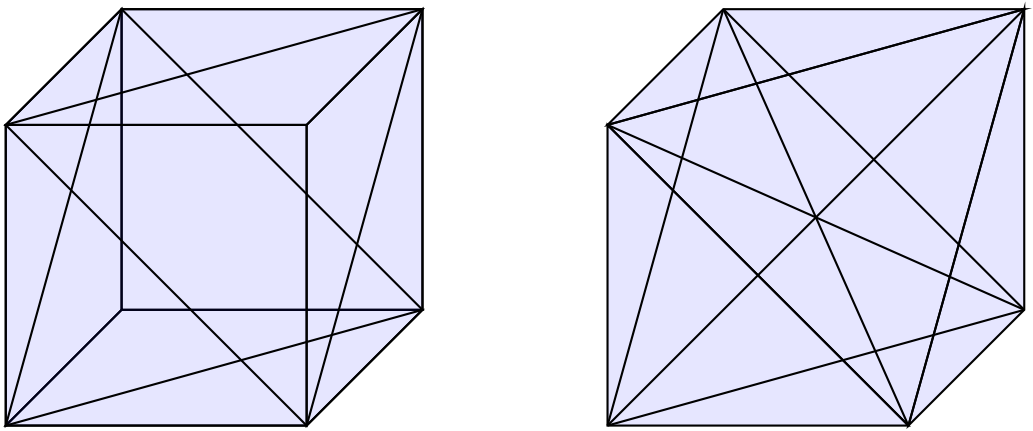

 \begin{remark} \label{rem:Sanchezcompare} 
 We see from the calculations above that $D(X_{4}) \not \cong D_{\Cox}(X_{4})$ as $ 44 = |\Theta| > 4! = \mathrm{rank}(K_0(X_{4}))$.
 However, the $G$-invariant full strong exceptional collection of line bundles on $D(X_{4})$ produced in \cite{sanchez2023derived} has similar features to $\Theta$, and it would be interesting to further explore their relationship (and also for larger $n$).
 Note also that Sanchez' collection is not simply a subset of $\Theta$.
 For example, some of the section polytopes corresponding to line bundles in Sanchez' collection are not simplices while all the $P_\theta$ are.
\end{remark}

Finally, let's study the action of $G \cong S_4 \times S_2$ on $\Theta$ in $D_{\Cox}(X_4)$. 
Let's begin by making a couple simple observations that distinguish this example from $X_3$.
First, the two zero-dimensional strata have different depths so both are fixed by $G$ and the depth and dimension semi-orthogonal decompositions do not coincide.
Further, neither the depth nor the dimension semi-orthogonal decompositions coincide with the $G$-orbits since among the top-dimensional strata, which all have depth $0$, are two standard simplices of volume $1/6$ which cannot be mapped to the simplices of volume $1/12$ inside the octahedron.
Turning back to our task, the entirety of the action can be understood geometrically from \Cref{fig:perm4}.
The $S_2$ factor acts on \Cref{fig:perm4} by reflection across the central point stratum of the octahedron.  
Transpositions $(ij)$ in the $S_4$ factor simply swap the $x_i$ and $x_j$ coordinates when $i,j \leq 3$. 
The remaining transpositions act by compositions of a reflection and a shear. 
For instance $(34)$ acts on \Cref{fig:perm4} by $\begin{pmatrix} 1 & 0 & 0 \\ 0& 1 & 0 \\ -1 & -1 & -1 \end{pmatrix}$.

As these examples illustrate, the Bondal-Thomsen collection and the Cox category for $X_{n+1}$ are combinatorially rich.
In fact, the Bondal-Thomsen elements for $X_{n+1}$ correspond to the strata of the periodic \emph{resonance arrangement}.
The fact that enumerating the chambers of the (unperiodized) resonance arrangement is a difficult open problem (see, e.g., \cite{kuhne2023universality}) suggests that explicitly describing $\Theta$ for $X_{n+1}$ may be difficult.
Nevertheless, it would be intriguing to study this combinatorics from the above perspective and perhaps address other particularly symmetric toric varieties as well.

\bibliographystyle{amsalpha}
\bibliography{refs}   

\Addresses
\end{document}